\pdfoutput=1
\documentclass[english]{article}
\usepackage[T1]{fontenc}
\usepackage[cp1251]{inputenc}
\usepackage{textcomp}
\usepackage{ifsym}
\usepackage{amsmath}
\usepackage{amssymb}
\usepackage{amsthm}

\usepackage{babel}
\usepackage{graphicx}
\usepackage[colorlinks=true]{hyperref}

\newtheorem*{theorem}{Theorem}
\newtheorem*{conjecture}{Conjecture}
\newtheorem{proposition}{Proposition}
\newtheorem*{lemma}{Lemma}

\theoremstyle{definition}
\newtheorem{definition}{Definition}
\author{Alexandra Skripchenko\footnote{This work is supported in part by RFBR (grant no. 10-01-91056), Russian Federation Govenment Grant no. 2010-220-01-077, Russian State Programme for the Support of Leading Scientific Schools (grant NSh-4995-2012.1) and GDRI France-Russie} \\ Moscow State University}
\title{On connectedness of chaotic sections of some 3-periodic surfaces}

\begin{document}
\maketitle
\begin{abstract}
In the present paper we construct a $\mathbb{Z}^3$-periodic surface in $\mathbb {R}^3$ whose almost all plane sections of a certain direction consist of exactly one connected component. This question originates from a problem of Novikov on the semiclassical motion of an electron in strong magnetic field. Our main tool is the Rips machine algorighm for band complexes. 
\end{abstract}

\section{Introduction}
A surface in $\mathbb {R}^3$ is called \emph{triply periodic} if it is invariant under translations by vectors from the lattice $\mathbb {Z}^3$. Regular plane sections of a triply periodic surface $\widehat M$ usually split into an infinite number of connected components some of which may be unbounded. The study of asymptotic behavior of the plane sections of periodic surfaces was initiated by S.P.Novikov (\cite{3}) in 1982 motivated by an application to the conductivity theory of monocrystals in magnetic field (see \cite{5} for details). By the physical nature of the problem the surface $\widehat M$ has to be a level surface of some smooth 3-periodic function. 

It was shown by A.V.Zorich (\cite{6}) and by I.A.Dynnikov(\cite{7}) that typically a regular plane section of a triply periodic surface either consists of compact components only or has unbounded components that have the form of finitely deformed periodic family of parallel straight lines. S.P.Tsarev constructed the first non-typical example, in which the unbounded components had an asymptotic direction but didn't fit into a strip of finite width (see \cite{7} for details). In Tsarev's case the plane direction is not ``totally irrational'' meaning that the the irrationality degree of this vector is 2. 

The presence of an asymptotic direction of the discussed curves is explained by the fact that the image of such a curve under the natural projection $\pi:\mathbb {R}^3\rightarrow\mathbb {T}^3 = \mathbb {R}^3/ \mathbb {Z}^3$ densely fills not the whole surface $M=\pi(\widehat M)$ but only a part that has genus one. 

\begin{definition}A plane section of the surface $\widehat M$ by a plane $\alpha$ is called \emph{chaotic} if it has at least one connected component such that the closure of its projection $\pi$ is a subsurface of $M$ (possibly with boundary) of genus strictly greater than one. 
\end{definition}

For studying the sections mentioned above it is natural to consider the foliation $F$ on $M$ defined by a restriction of the following $1$-form $\omega$ with constant coefficients on $M$: $\omega=H_{1}dx^{1}+H_{2}dx^{2}+H_{3}dx^{3}$, where vector $(H_{1},H_{2},H_{3})$ determines our plane direction. The main tool for studying the measured foliations on surfaces is interval exchange transformations which arise as the first return map on a transversal for Hamiltonian flows on surfaces. However, in chaotic case we deal with a very particular case of such flows on surface, because the genus of the surface is equal to $3$, and our $1$-form $\omega$ has only three independent integrals. 

The first example of a chaotic section was constructed by I.A.Dynnikov in \cite{7}. The section $\alpha\cap\widehat M$ in Dynnikov's example is in a sense self-similar, which suggests that it may consist of a single connected component wandering over the whole plane. In \cite{1} Dynnikov formulated the following conjecture:
\begin{conjecture}In the case of genus three almost any chaotic section $\alpha\cap\widehat M$ consists of exactly one connected curve. . 
\end{conjecture}
In the present paper we prove this conjecture for two examples of the chaotic sections. More precisely, we prove the following result:
\theoremstyle{theorem}
\begin{theorem}There exist a triply periodic surface $\widehat M$ and a vector $H$ such that the section of $\widehat M$ by almost any plane orthogonal to $H$ consists of exactly one connected component.
\end{theorem}
\theoremstyle{remark}
\newtheorem{remark}{Remark}
\begin{remark}We construct piecewise smooth surface which can be smoothed after a finite deformation of the whole picture(more precise description of this deformation is provided in section 3). See Figure \ref{0}, where such intersections are drawn by bold lines.
\end{remark}
\begin{figure}
\includegraphics[width=12cm,height=10cm]{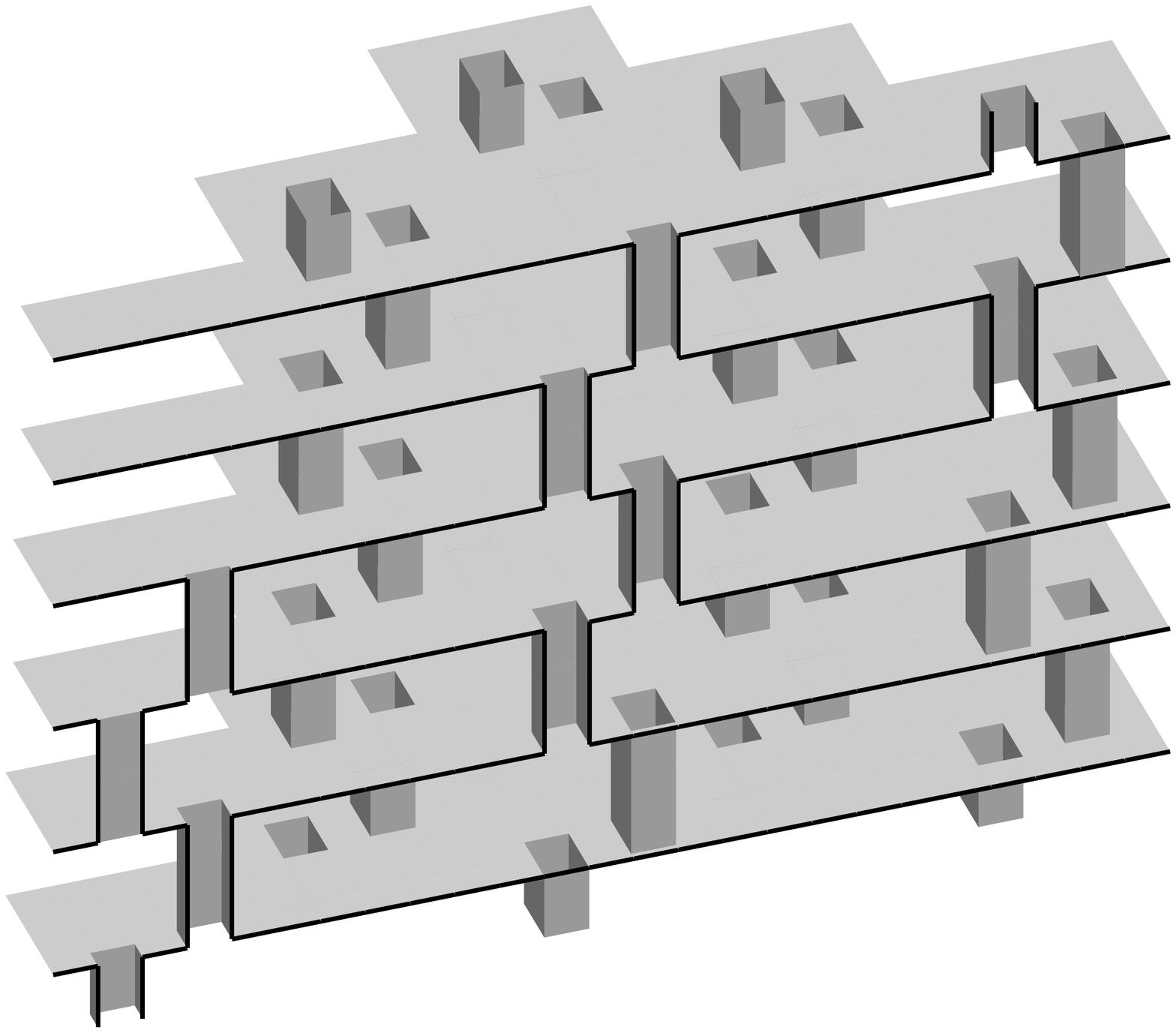}
\caption{Section of a 3-periodic surface by a plane}
\label{0}
\end{figure}
\begin{remark}We consider two examples of chaotic sections, in one of which the surface is a level surface of an even fuction which is always the case in all physically meaningful examples because surface $\widehat M$ is a Fermi surface of some metal.
\end{remark}

The paper is organized as follows. In section 2 we introduce our main tool --- interval identification systems (IIS) of thin type. IIS is a kind of generalization of interval exchange transformations and interval translation mappings. We describe an analogue of the Rauzy induction for these systems and recall the definition of IIS of thin type (thin property means that we can apply the Rauzy induction to our system infinitely many times, see precise definition below). We also recall the main idea of construction of two examples of IIS of thin type (one of them has additional symmetry): both of our systems are in some sense self-similar. The same idea is often used for construction of examples of pseudo-Anosov diffeomorphisms.

\noindent In section 3 we construct an example of chaotic section using interval identification system of thin type. Our construction is almost the same as one from \cite{1}.

\noindent In section 4 we prove our main result. In order to do this, we consider IIS as a particular case of band complex that is the notion from the theory of $\mathbb R$-trees, and describe the Rips machine algorithm for these band complexes. 

\section{Interval Identification Systems}
In \cite{1} Dynnikov reduced the question about chaotic sections in genus three to the study of \emph{interval identification systems}. In the current section we briefly recall main definitions and some results from \cite{1} and \cite{2} that will be used in proof of our main result. 

The notion of interval identification system was introduced by Dynnikov and B. Wiest in \cite{8} and studied then by Dynnikov in \cite{1}. 

\begin{definition}
An \emph {oriented interval identification system} is an object that consists of:
\begin{enumerate}
\item An interval $\left[A,B\right]\subseteq\mathbb R$ with $A<B$ (we call this
interval \emph{the support interval});
\item A natural number $n$ (we call this number \emph{the order} of the system);
\item A collection of $n$ unordered pairs $\left\{ \left[a_{i},b_{i}\right],\left[c_{i},d_{i}\right]\right\} $
of subintervals of $\left[A,B\right]$ in each of which the intervals
have equal lengths: $b_{i}-a_{i}=d_{i}-c_{i}>0$.
\end{enumerate}
\end{definition}
For every pair of intervals $\left\{ \left[a_{i},b_{i}\right],\left[c_{i},d_{i}\right]\right\} $
from an interval identification system we consider the orientation preserving affine isometry between them and we will say that a point $x$ of $\left[a_{i},b_{i}\right]$ is identified to a point $y$ of $\left[c_{i},d_{i}\right]$(and write $x\leftrightarrow_{i}y$)
if $x$ is mapped to $y$ or $y$ is mapped to $x$ under this isometry.
So we write $x$$\leftrightarrow_{i}$$y$ if there exists $t\in\left[0,1\right]$
such that $\left\{ x,y\right\} =\left\{ a_{i}+t\left(b_{i}-a_{i}\right),c_{i}+t\left(d_{i}-c_{i}\right)\right\} $.
A more general object, interval identification system, in which some pairs of intervals are identified by orientation reversing maps, is not considered in this paper. All interval identification systems in this paper are assumed to be oriented.

Interval identification systems are a natural generalization of interval exchange transformations and interval translation mappings. Similar objects also have appeared in the theory of {$\mathbb R$}-trees (sometimes without giving them specific name) as an instrument for describing the leaf space of a band complex (see section 3 for details).

\begin{definition}
An interval identification system is called \emph{balanced}, if
$A=\min_{i}(a_{i}), B=\max_{i}(b_{i})$ and ${\textstyle \sum_{i=1}^{n}\left(b_{i}-a_{i}\right)}=B-A$.
\end{definition}

\begin{definition}
An interval identification system is called \emph{symmetric} if $a_{i}-A$$=$ $B$$-d_{i}$ for each of
$i=1,\ldots,n$.
\end{definition}

In the current paper we consider only oriented interval identification systems of order 3. With each interval identification system
$$S=\left(\left[A,B\right];\left[a_{1},b_{1}\right]\leftrightarrow\left[c_{1},d_{1}\right];\left[a_{2},b_{2}\right]\leftrightarrow\left[c_{2},d_{2}\right];\left[a_{3},b_{3}\right]\leftrightarrow\left[c_{3},d_{3}\right]\right)$$
we associate a graph $\Gamma\left(S\right)$ whose vertices are all points of the support interval, and two vertices of the graph are connected by an edge if and only if these two
points are identified by our system in the sense that is described
above. The system $S$ determines an equivalence relation
$\sim$ on the support interval $\left[A,B\right]$: the points lying
on the same connected component of the graph $\Gamma\left(S\right)$ are
said to be \emph{equivalent}. The set of points equivalent in this sense
to $x$ is called \emph{the orbit} of $x$ in S. The connected component of our graph that contains a vertex $x\in [A,B]$ will be denoted by $\Gamma_{x}(S)$.

To study such properties of orbits of an interval identification systems as finiteness and being everywhere dense it is convenient to use a special Euclid type algorithm. Analog of this process that appears in the theory of interval exchange transformation is called the \emph{Rauzy induction}. The main idea is that from any interval identification system one constructs a sequense of interval identifiacation systems equivalent in a certain sense to the original one but with a smaller support. Equivalence of two IIS here means that systems have the same behavior of orbits (see precise definition below and Figure \ref{1}, where all three interval identification systems are equivalent). Combinatorial properties of this sequence are responsible for ``ergodic'' properties of the original interval identification system.

\begin{definition}
Two interval identification systems $S_{1}$ and $S_{2}$ with supports $[A_{1},B_{1}]$ and $[A_{2},B_{2}]$, respectively, are called \emph{equivalent}, if there is a real number $t \in {\mathbb R}$ and an interval $[A,B] \subset [A_{1},B_{1}] \cap  [A_{2}+t,B_{2}+t]$ such that
\begin {enumerate}
\item every orbit of each of the systems $S_{1}$ and $S_{2}+t$ contains a point lying in $[A,B]$
\item for each point $x\in [A,B]$ the graphs $\Gamma_{x}(S_{1})$ and $\Gamma_{x}(S_{2}+t)$ are homotopy equivalent through mappings that are identical on all vertices which belong to $[A,B]$ and such that the full preimage of each vertex contains only finitely many vertices of the other graph.
\end {enumerate}
\end{definition}

\noindent It is easy to see that it is an equivalence relation.

\emph{The Rauzy induction} for an interval identification system is a recursive application of admissible transmissions followed by reductions as described below.
Like in case of interval exchange transformation, application of admissible transmission is possible in the situation when two subintervals of interval identification system have the same right (or left) end which coincides with the end of the support interval (let us call these two subinterval ``the biggest one'' and ``the smallest one''); in this case we can shift the smallest of these two subintervals to left (right) in such way that the right (left) end of this subinterval will coincide with the rightmost (leftmost) point of the pair of the biggest subinterval. Application of reduction is possible in the situation when a rightmost (or leftmost) part of our support interval is covered by only one subinterval from our system; in this case we can remove this part from the support interval. An example of an iteration (transmission on the right plus reduction on the right) of the Rauzy induction to a symmetric interval identification system is shown in Figure \ref{1}, precise definitions are provided below.
 
\begin{definition}
Let $$S=\left(\left[A,B\right];\left[a_{1},b_{1}\right]\leftrightarrow\left[c_{1},d_{1}\right];\left[a_{2},b_{2}\right]\leftrightarrow\left[c_{2},d_{2}\right];\left[a_{3},b_{3}\right]\leftrightarrow\left[c_{3},d_{3}\right]\right)
$$
be an interval identification system and let one of the subintervals,
$\left[c_{1},d_{1}\right]$, say, be contained in another one $\left[c_{2},d_{2}\right]$, say. Let $S^{'}$ be the interval
identification system obtained from $S$ by replacing the pair $\left[a_{1},b_{1}\right]\leftrightarrow\left[c_{1},d_{1}\right]$
with $\left[a_{1},b_{1}\right]\leftrightarrow\left[c_{1}^{'},d_{1}^{'}\right]$
where $\left[c_{1}^{'},d_{1}^{'}\right]$=$\left[c_{1},d_{1}\right]-c_{2}$+$a_{2}$$\subset$$\left[a_{2},b_{2}\right]$.
We say that $S^{'}$is obtained from $S$ by a \emph{transmission} (of $[c_{1},d_{1}]$ along $[a_{2},b_{2}]\leftrightarrow [c_{2},d_{2}]$).

If, in addition, we have $c_{2}=A$, then this operation is called
an \emph{admissible transmission on the left, }and if $d_{2}=B$,
an \emph{admissible transmission on the right}.
\end {definition}

\begin{definition}
Let
$$
S=\left(\left[A,B\right];\left[a_{1},b_{1}\right]\leftrightarrow\left[c_{1},d_{1}\right];\left[a_{2},b_{2}\right]\leftrightarrow\left[c_{2},d_{2}\right];\left[a_{3},b_{3}\right]\leftrightarrow\left[c_{3},d_{3}\right]\right)
$$
be an interval identification system and let $d_{1}=B$. We call all
endpoints of our subintervals \emph{critical points}. Assume that
the point $B$ is not covered by any point of intervals from $S$ except $d_{1}$
and that the interior of the interval $\left[c_{1},d_{1}\right]$ contains
a critical point. Let $u$ be the rightmost critical point. Then the interval $[u,B]$ is covered by only one interval from our system. Replacing the pair $\left[a_{1},b_{1}\right]\leftrightarrow\left[c_{1},d_{1}\right]$
with $\left[a_{1},b_{1}-d_{1}+u\right]\leftrightarrow$$\left[c_{1},u\right]$
in $S$ with simultaneous cutting off the part $[u,B]$ from the support interval will be called a\emph{ reduction on the right} (of the pair $\left[a_{1},b_{1}\right]\leftrightarrow\left[c_{1},d_{1}\right]$). A reduction
on the left is defined in the symmetric way.
\end{definition}

\begin{figure}
\includegraphics[width=15cm,height=10cm]{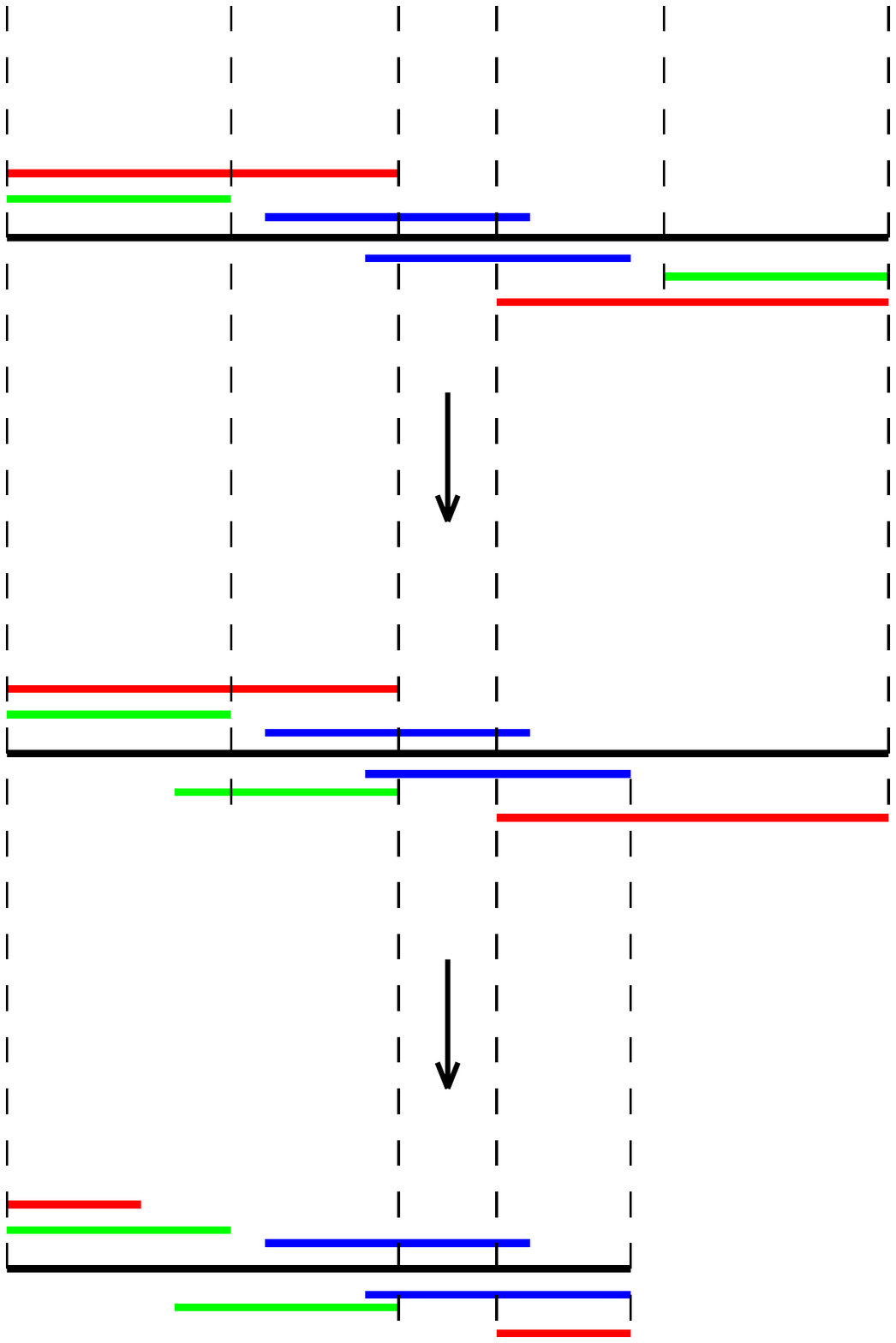}
\put(-395,235){$a_{1}=a_{2}$}
\put(-328,235){$b_{2}$}
\put(-276,235){$b_{1}$}
\put(-250,235){$c_{1}$}
\put(-207,235){$c_{2}$}
\put(-155,235){$d_{1}=d_{2}$}
\caption{An iteration of the Rauzy induction: transmission of $[c_{2},d_{2}]$ on the right + reduction of $[c_{1},d_{1}]$ on the right}
\label{1}
\end{figure}

In \cite{2} the Rauzy induction was used for the construction of symmetric interval identification system of a \emph{thin type}. By the latter we mean an interval identification system for which an equivalent system may have arbitrarily small support. In \cite{9} such interval translation mappings are called ITM of infinite type. Thin case in the theory of {$\mathbb R$}-trees was discovered by
G. Levitt in \cite{10}.

Let us recall the constructions of interval identification systems of thin type from \cite{2} and \cite{1}. We denote the following matrices:

$$N_{1}=\begin{pmatrix}
3 & 1 & -1 & -4\\
-1 & 2 & 0 & 0\\
-2 & -2 & 1 & 4\\
3 & 2 & -1 & -5\end{pmatrix}$$

and

$$N_{2}=\begin{pmatrix}
-2 & 2 & 1 & 0 & 1\\
2 & -5 & -2 & 3 &-2\\
1 & 0 & 0 & -1 & 0\\
1 &-2 & -1 & 1 & 0\\
0 & -2 & -2 & 3& -2 \end{pmatrix}.$$

\noindent It is easy to see that each matrix has exactly
one real positive eigenvalue $\lambda_{1}<1$ and $\lambda_{2}<1$. Their approximate values
are $\lambda_{1}\approx0.254$ and $\lambda_{2}\approx0.0797,$ respectively.

In \cite{2} the following result was proved:
\begin{proposition}
Let $\left(a,b,c,u\right)$ be an eigenvector
of the matrix $N_{1}$ with the eigenvalue $\lambda_{1}$ and positive coordinates.
Then the corresponding symmetric interval identification system
\begin{equation*}
\begin{split}
S_{1}=(\left[0,a+b+c\right];& \left[0,a\right]\leftrightarrow\left[b+c,a+b+c\right],\\
                    & \left[0,b\right]\leftrightarrow\left[a+c,a+b+c\right], \\
                    & \left[u,u+c\right]\leftrightarrow\left[a+b-u,a+b+c-u\right])
\end{split}
\end{equation*}
is of thin type. Approximate values of $\left(a,b,c,u\right)$ normalized by $a+b+c=1$ are equal to $\left(0.444,0.254,0.302,0.292\right)$.
\end{proposition}

The main idea of proof is standard and originates from the construction of pseudo-Anosov map on surface provided by W. Veech (see \cite{12} for details). Matrix $N_{1}$ is responsible for the behavior of parameters during the application of the Rauzy induction. One can check that $N_{1}^{-1}$ is an integer matrix with non-negative entries which become positive for sufficient power of the matrix, and an interval identification system, obtained after 6 iterations of the Rauzy induction from $S_{1}$, is a scaled down version of the original one and $\lambda_{1}$ is a coefficient of contraction. So, $\lambda_{1}$ is a Perron-Frobenius eigenvalue of $N_{1}$. 

Similar proposition was proved in \cite{1}.
\theoremstyle{theorem}
\begin{proposition}
Let $\left(a,b,c,d,e\right)$ be an eigenvector
of the matrix $N_{2}$ with the eigenvalue $\lambda_{2}$ and positive coordinates.
Then the corresponding interval identification system
\begin{equation*}
\begin{split}
S_2=(\left[0,a+b+c\right];& \left[0,a\right]\leftrightarrow\left[b+c,a+b+c\right],\\
                    & \left[0,b\right]\leftrightarrow\left[a+c,a+b+c\right], \\
                    & \left[d,d+c\right]\leftrightarrow\left[e,e+c\right])
\end{split}
\end{equation*}
is of thin type. Approximate values of $\left(a,b,c,d,e\right)$ normalized by $a+b+c=1$ are equal to $\left(0.4495,0.2943,0.2562,0.4292, 0.0898\right)$.
\end{proposition}

In order to check it, one can prove that that ten admissible
transmissions on the right with subsequent reductions on the right result in the same system but
$\lambda_{2}$-fold contracted.

\section{Plane Sections of Triply Periodic Surfaces}
Now we explain how to construct a chaotic section using interval identification system of thin type. This construction for the IIS $S_{2}$ was described in \cite{1}. Here we deal with the interval identification system $S_{1}$. 

We construct a piesewise smooth surface in the 3-torus $\mathbb T^{3}$ and consider asymptotic behavior of sections of $\mathbb {Z}^3$-covering of this surface in $\mathbb {R}^3$ by a family of parallel planes $\alpha: H_{1}x_{1}+H_{2}x_{2}+H_{3}x_{3} = const$,
where $H$ is some fixed covector.

For technical reasons we will vary not the covector $H$ but the coordinate system and fundamental domain of the lattice in $\mathbb R^{3}$, so as to have the
coordinates of $H$ constant and equal to $(0,1,0)$. We start with our four parameters ($a,b,c,u$) which were specified previously in definition of $S_{1}$. Let us introduce the following notation for rectangles in the plane $\mathbb R^{2}$:
\begin{align}
T_{1}& =[0,1]\times[0,a+b+2c] \\
T_{2}& =[1/5,2/5]\times[u,u+c] \\
T_{3}& =[3/5, 4/5]\times[a,a+c] \\
T_{4}& =[1/5,2/5]\times[a+b-u,a+b+c-u]. 
\end{align} 

\noindent One can check that $T_{2},T_{3},T_{4}\subset[0,1]\times[0,1]$. As a fundamental domain $M_{0}$ of the surface $\widehat M$, we take the following piecewise linear surface (see Figure \ref{P}):
$$(T_{1}\setminus(T_{2}\cup T_{3})) \times{1/4}  \cup(T_{1}\setminus (T_{3} \cup T_{4)})\times{3/4} \cup\partial T_{2} \times[0,1/4] \cup\partial T_{3} \times[1/4,3/4]\cup\partial T_{4} \times[3/4,1].$$
\begin{figure}
\includegraphics[width=15cm,height=15cm]{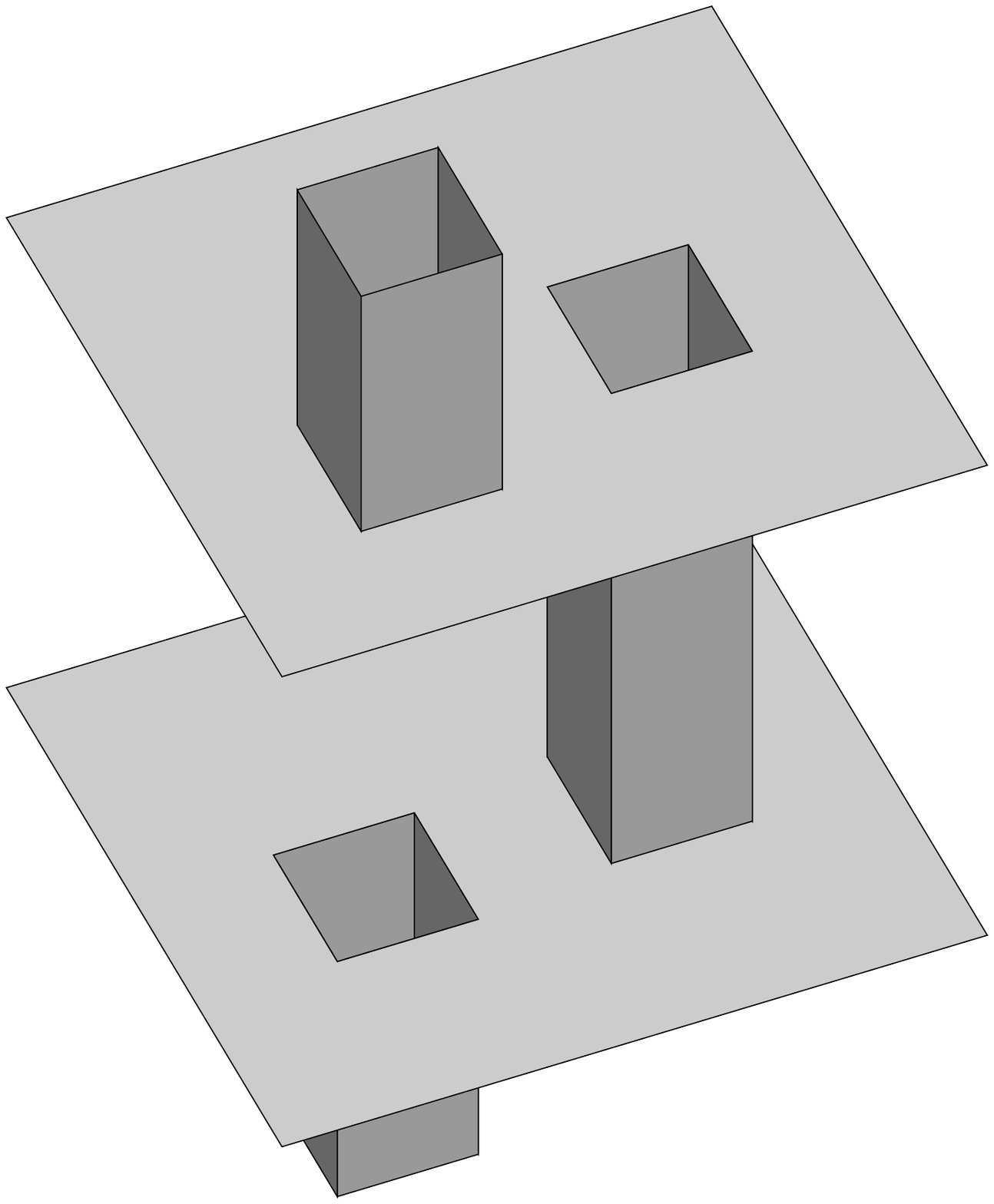}
\caption{Fundamental domain of surface M}
\label{P}
\end{figure}
As a fundamental domain we take the lattice spanned by the following three vectors:
$$e_{1}=(1,-b-c,0), e_{2}=(1,a+c,0), e_{3}=(0,a+b-2u,1).$$
The covering surface $\widehat M$ is equal to $M_{0}+G$ where $G$ is a translation group based on our lattice. As far as for covector $H$, we always take $H=(0,1,0)$.

Let us consider a band complex $X$ associated with our interval idenfication system. This band complex consists of three bands, each band is glued up to our support interval, and orientation preserving affine isometries from the definition of our system determine generalized leaves of the lamination associated with band complex (see \cite{4} for precise definitions). Bases of each band coincide with corresponding pairs of intervals (see a picture in Figure \ref{C}).

\begin{figure}
\includegraphics[width=15cm,height=15cm]{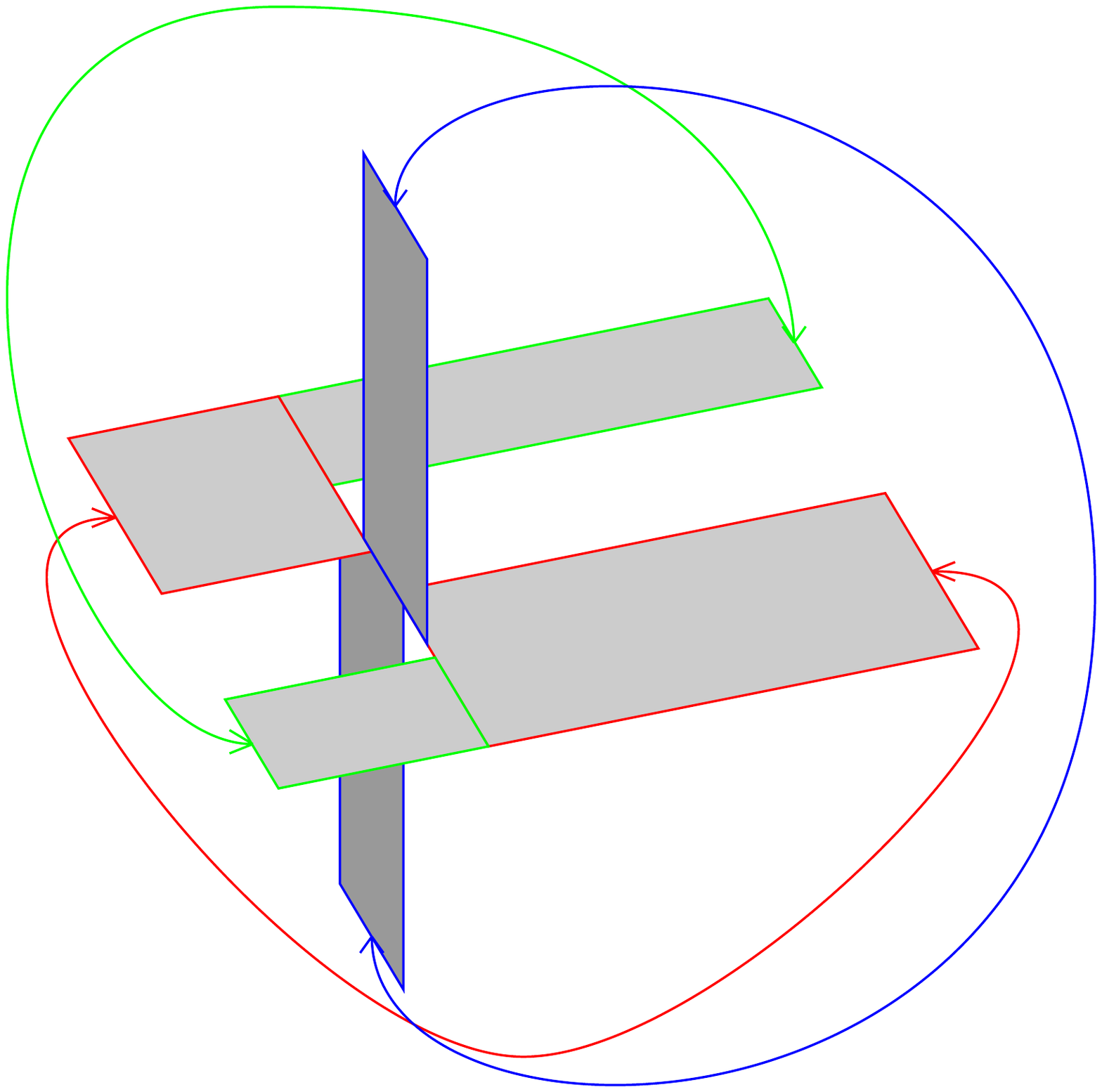}

\caption{Band complex}
\label{C}
\end{figure}
Now we can prove the following proposition.
\begin{proposition}
Let $\left(a,b,c,u\right)$ be an eigenvector
of the matrix $N_{1}$ with the eigenvalue $\lambda_{1}$ and positive coordinates. Then sections of the surface $\widehat M$, constructed as above with these values of the parameters, by any plane orthogonal to $H=(0,1,0)$ are chaotic.
\end{proposition}
\begin{proof}
Let us denote by $M$ the image of the projection of $\widehat M$ in torus: $M=\pi(\widehat M)$. For studying the sections $\alpha\cap\widehat{M}$ we consider the foliation $F$ on $M$ defined by a restriction of the 1-form $H_{1}dx^{1}+H_{2}dx^{2}+H_{3}dx^{3}$ with the constant coefficients to $M$. It is easy to see that the surface $M$ has genus $3$. We need to show that the foliation $F$ is minimal, that is the closure of any leaf of $F$ coincides with $M$. The leaves of this foliation are the images of the sections under consideration under the projection $\pi$. We need to check that the resulting foliation $F$ does not have closed leaves and saddle connection cycles.  Taking into account, that our thin interval identification system is minimal (see \cite{1} for details of proof), it is enough to ensure that there are no closed leaves or saddle connections between distinct saddles. It can be directly checked that all saddles lie in different planes of the form $x^{2}=const$; hence saddle connections between them are impossible.

In order to show that $F$ does not have closed leaves, we consider not the surface M itself but one
of the two parts into which it cuts the torus $\mathbb T^{3}$. Both parts are filled handlebodies of genus $3$. We denote one of them (which contains a point $\pi(0,0,1/2)$) by $M_{1}$ and $\mathbb Z^{3}$-covering of $M_{1}$ by $\widehat{M_{1}}$. One can check directly that for any plane $\Pi$ defined by an equation of the form $x^{2}=const$ the section $\Pi\cap \widehat M_{1}$ has $\Pi\cap \widehat X$ as a deformation retract and the restriction of the form $\omega=dx^{2}$ to $X$ defines a lamination associated to the band complex $X$ (see Figure \ref{S}). $\widehat X$ is an abelian universal cover of $X$.
\begin{figure}
\includegraphics[width=14cm,height=15cm]{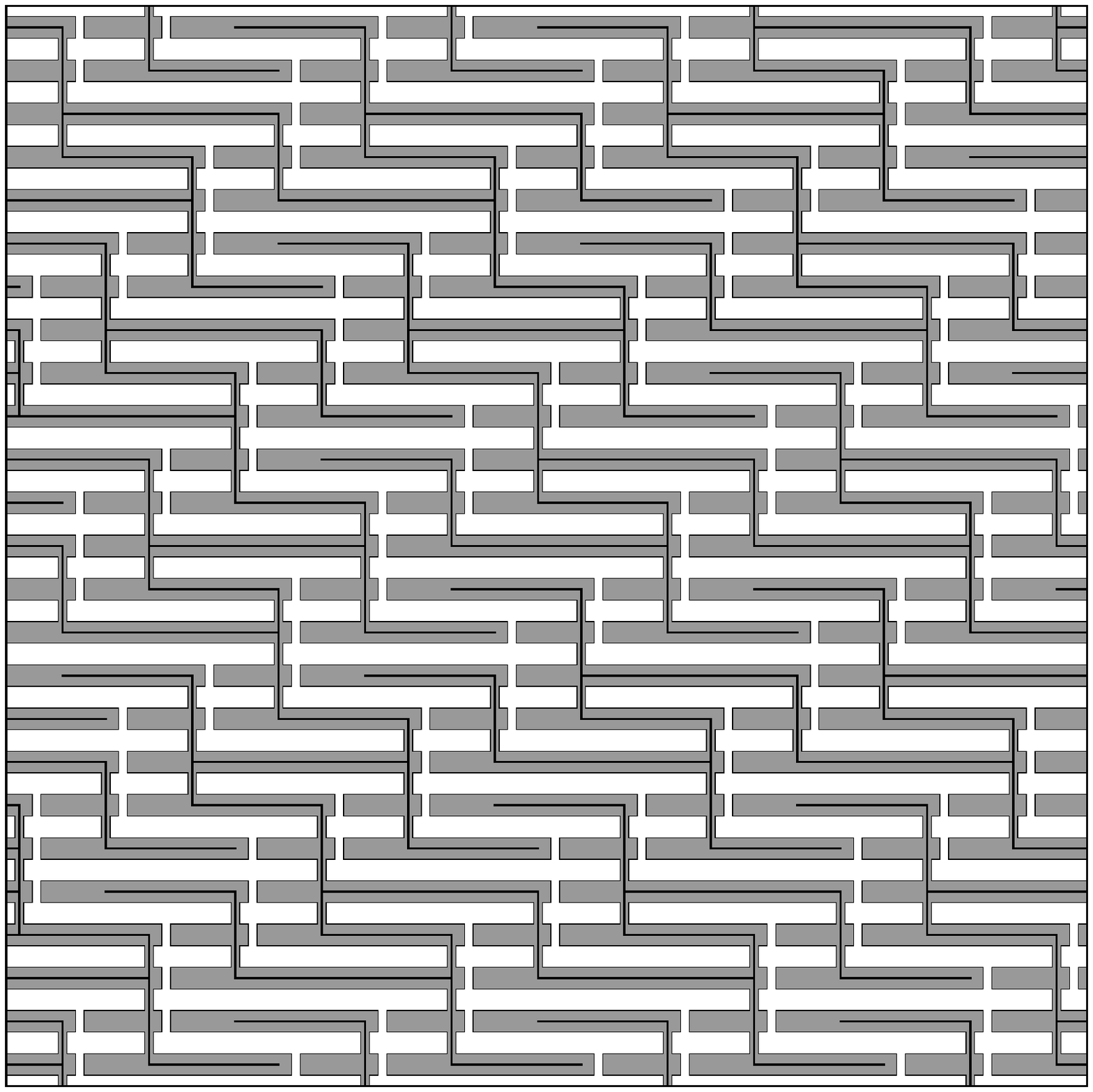}
\caption{Section of the submanifold $\widehat{M_{1}}$ and universal cover of band complex}
\label{S}
\end{figure}

The foliation $F$ has closed leaves if and only if the sections of the manifold with boundary $\widehat{M_{1}}$ by planes $x^{2}=const$ have either compact or non-simply-connected regular components. Hence the same have to be true after replacing $\widehat{M_{1}}$ on $\widehat{X}$. Furthermore, the latter can be reformulated in terms
of the system $S_{1}$ by saying that it must have an essential set of finite orbits and an essential set
of non-simply-connected ones, but we know that $S_{1}$ doesn't have finite orbits.
\end{proof}

\noindent As it was mentioned before our surface can be smoothed so that all $H$-sections will be just finitely deformed. In order to check this it is enough to ensure that the saddle singularities of the foliation $F$ which are the tangency points of the surface and the planes $x^2=const$ appear on the four rectangulars that are parallel to these planes. 

\section{Main Result}
\label{3}
In this section we prove our main theorem. As it was shown in \cite{1}, the question about number of connected components of chaotic section can be reduced to the studying of topological structure of orbits of corresponding interval identification systems. More precisely, as it was proved independently by Dynnikov in \cite{1} and D. Gaboriau in \cite{11}, for IIS of thin type almost all graphs $\Gamma_{x}$ are infinite trees. So, the same is true for our 
interval identification systems $S_{1}$ and $S_{2}$, and in order to estimate the number of connected components of our chaotic sections we need to know how many topological ends each of these graphs have. It was observed in \cite{1} that the property of the IIS to have almost all graphs $\Gamma_{x}$ with exactly one topological end implies that for the triply periodic surface $\widehat M$ and vector $H$ constructed in previous part section of $\widehat M$ by almost any plane orthogonal to $H$ consists of exactly one connected curve. We explain this connection in lemma below.

Similar questions number of topological ends for objects related to ours were also discussed in \cite{4} in terms of leaves of the lamination associated to the band complex and in \cite{11} in terms of Cayley graphs.  
We use the terminology from \cite{4}: a \emph{1-ended tree} is a locally compact graph obtained from a ray by attaching to it infinitely many finite trees without a uniform bound on their diameter. Similarly, a \emph{2-ended tree} is obtained from the above definition by replacing the word "ray" by the word "line".

Taking into account that in our case the number of vertices of valence one is equal to number of vertices of valence three one can check that for almost all graphs $\Gamma_{x}$ number of topological ends can be equal to only one or two (see also \cite{11}, where this fact is proved for more general case). 

Let us prove the following result:

\theoremstyle{theorem}

\begin{proposition} Let $\left(a,b,c,u\right)$ be an eigenvector
of the matrix $N_{1}$ with the eigenvalue $\lambda_{1}$ and positive coordinates.
Then for the corresponding symmetric interval identification system
\begin{equation*}
\begin{split}
S_{1}=(\left[0,a+b+c\right];& \left[0,a\right]\leftrightarrow\left[b+c,a+b+c\right],\\
                    & \left[0,b\right]\leftrightarrow\left[a+c,a+b+c\right], \\
                    & \left[u,u+c\right]\leftrightarrow\left[a+b-u,a+b+c-u\right])
\end{split}
\end{equation*}
almost all graphs $\Gamma_{x}$ have only one topological end. 
\end{proposition}

It was observed in \cite{1} that this property implies that for the triply periodic surface $\widehat M$ and vector $H$ constructed in previous part section of $\widehat M$ by almost any plane orthogonal to $H$ consists of exactly one connected curve. We explain this connection in Lemma.

\begin{proof}
Informally, the number of topological ends of an infinite tree is the number of connected components on infinity. So let us consider the following process: on each step we find vertices of valence one of our graph and remove them with corresponding edges (including ends which are points of the support interval) from the graph. As it was shown in \cite{4}, in order to prove our statement it is enough to check that after infinite number of iterations of such process we remove almost all points of our support segment. For proof of that we consider our IIS as a particular case of a band complex with vertical lamination and use the Rips machine algorithm for band complexes. The Rips machine is also a Euclid type algorithm with the same purpose as the Rauzy induction: it allows us to construct a band complex  with the same behavior of leaves of the lamination as original one has, but with smaller support. The main difference between two algorithms for thin band complexes is that application of the Rips machine consists only in reductions (without translations) but it is possible to remove any (not only rightmost or leftmost) part of the support interval which is covered by only one base of our band complex. We proceed with precise description.

In general case application of the Rips machine to a band complex cosists in consecutive application of six moves which are geometric version of moves of Razborov and Makanin. These moves produce from $X$ another band complex $X'$ which is equivalent to $X$ (or, more precisely, if $X$ and $X'$ are related by moves, then the leaf spaces of their universal covers equipped with pseudometric are isomorphic). The complete list of moves is provided in \cite{4}, but for band complexes of thin type only one type of such moves is eventually applied. This move ($M5$ in terms of \cite{4}) is called \emph{collapse from a free subarc}. 

A subarc $J$ of a base is said to be \emph{free} if $J$ has positive measure and the interior of $J$ meets no other base. Assume that $J$ is maximal free subarc of some base (say, $[x,y]\times 0$) of a band (say, $[0,a]\times[0,1]$). The move consists of collapsing $J\times[0,1]$ to $J\times1 \cup FrJ\times[0,1],$ where $FrJ$ means boundary of $J$. Typically, the band is replaced by two new bands (see Figure \ref {55}).
\begin{figure}
\includegraphics[width=10cm,height=8cm]{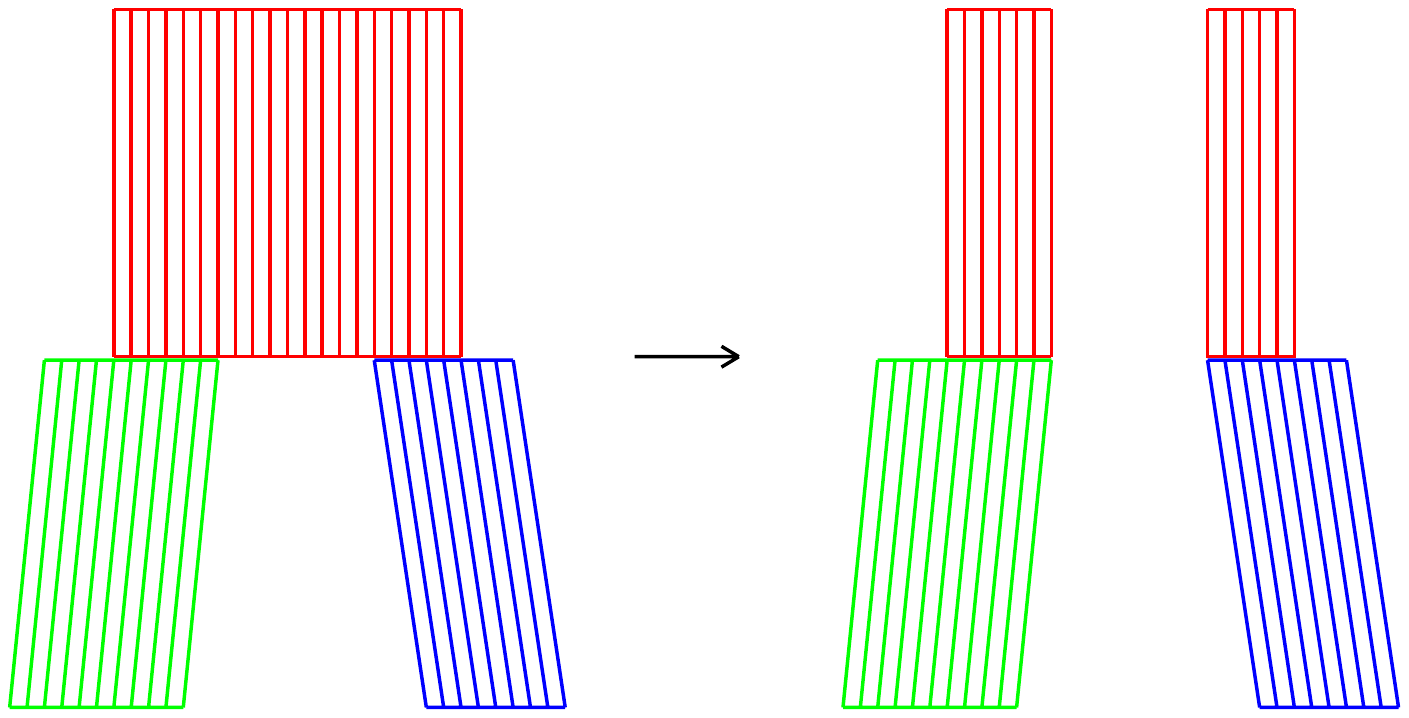}
\put (-190,100){J}
\caption{Collapse from a free subarc}
\label{55}
\end{figure}

In a band complex it might happen that several bands are glued together in such way that they form a long band. In other words, we say that a sequence $B_{1},B_{2},...,B_{n}$ forms a \emph{long band} provided that:
\begin{itemize}
\item the top of $B_{j}$ is identified with the bottom of $B_{j+1}$ and meets no other bands for $j=1,2,...,n-1$, and
\item the sequence of bands is maximal with respect to these properties.

\noindent In the output of one iteration of the Rips machine algorithm long bands are treated as units and we can replace any long band by a singl band with bases of bottom $B_{1}$ and top $B_{n}$ and length which is equal to sum of lengths of all $n$ bands.
\end{itemize}

Let us check what happens with a band complex $S_{1}$ (we use the same notation for IIS and corresponding band complex) under the action of the Rips machine. As it was shown in Section 2, $\lambda_{1}$ is an algebraic number ---  the eigenvalue of an integer matrix, and $a, b, c, u$ are components of the eigenvector which corresponds to $\lambda_{1}$ such that normalizing condition $a+b+c=1$ holds. One can check that these numbers can be expressed by polynomials of $\lambda_1$ and rational numbers in the following way:

\begin{align}
a& =2\lambda_{1} - \lambda_{1}^{2}\\
b& =\lambda_{1}\\
c& =\lambda_{1}^{2} - 3\lambda_{1} + 1\\
u& =\lambda_{1}^{3}/4 - 3\lambda_{1}^2/2 + 5\lambda_1/2 - 1/4
\end{align}

After $4$ first iteration of the Rips machine we obtain the band complex with three bands. Bases of these bands are denoted by $a',b',c',u'$ and can be expressed in terms of bases of the original complex in the following way:
$$\begin{pmatrix}
a'\\b'\\c'\\u'
\end{pmatrix}
=
\begin{pmatrix} 
-4&4&1&2\\
-1&2&0&0\\
2&0&-1&-2\\
-1&3&0&-1
\end{pmatrix}
\begin{pmatrix}
a\\b\\c\\u
\end{pmatrix}.
$$

After the next iteration of the Rips machine from our band complex $X$ we obtain the band complex that consists of two support intervals with four bands attached to them. Let us denote width of bases of such bands by $r_{1}$ and $r_{2}$ for two red bands (from left to right), $g$ for the green band and $n$ for the blue band. The length of the subinterval of support interval between two red intervals will be denoted by $h$. Vertical lengths of corresponding bands will be denoted by $l_{1}$ and $l_{2}$ for two red bands, $l_{3}$ for the green band and $l_{4}$ for the blue band. One can check that $l_{1}=2, l_{2}=1, l_{3}=5, l_{4}=5$ and $$\begin{pmatrix}
r_{1}\\r_{2}\\h\\g\\n
\end{pmatrix}
=
\begin{pmatrix} 
1&-1&-1&0\\
-1&0&2&2\\
0&1&0&-1\\
0&1&0&0\\
0&0&1&0
\end{pmatrix}
\begin{pmatrix}
a'\\b'\\c'\\u'
\end{pmatrix}.
$$ This band complex will be denoted by $Y$. The scheme of these 5 steps of the Rips machine is shown in Fugure \ref{2} and band complex $Y$ is shown in the top of Figure \ref{3}.
\begin{figure}
\includegraphics[width=20cm,height=14cm]{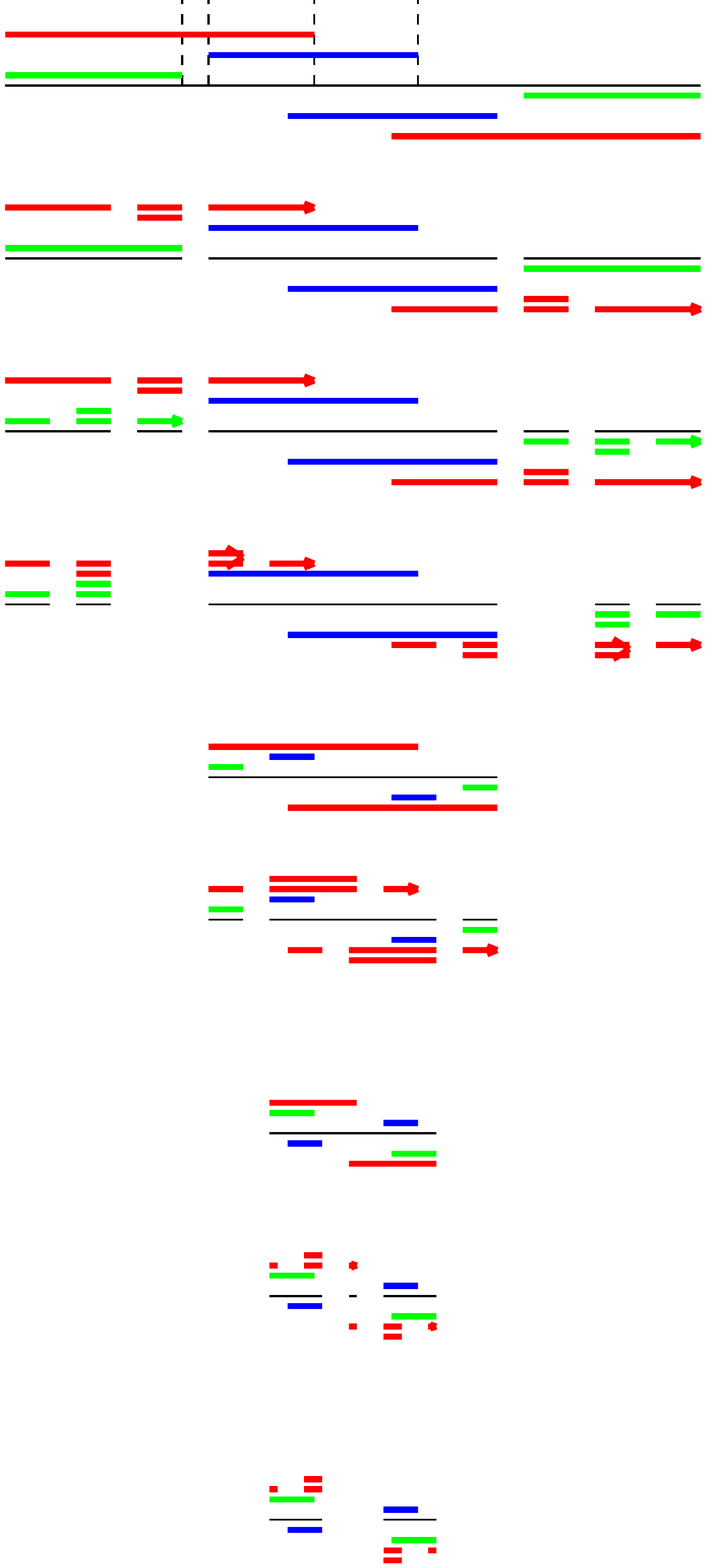}
\put(-410,365){\vector(0,-10){10}}
\put(-410,320){\vector(0,-10){10}}
\put(-410,275){\vector(0,-10){10}}
\put(-409,235){\line(0,-10){10}}
\put(-411,235){\line(0,-10){10}}
\put(-410,195){\vector(0,-10){10}}
\put(-411,150){\line(0,-10){10}}
\put(-409,150){\line(0,-10){10}}
\put(-410,100){\vector(0,-10){10}}
\put(-411,50){\line(0,-10){10}}
\put(-409,50){\line(0,-10){10}}
\put(-482,400){$b$}
\put(-470,400){$u$}
\put(-420,400){$a$}
\put(-385,400){$u+c$}
\caption{First Steps of the Rips Machine: Example 1}
\label{2}
\end{figure}

\noindent One can check that 
\begin{align}
r_{1}& =\lambda_{1}^{3}\\
r_{2}& =-2\lambda_{1}^{3}+5\lambda_{1}^{2}-5\lambda_{1}+1\\
h& =\lambda_{1}^{3}/4-3\lambda_{1}^2/2+3\lambda_{1}/2-1/4\\
g& =\lambda_{1}^2\\
b& =-\lambda_{1}^{3}/2+\lambda_{1}^2-2\lambda_{1}+1/2
\end{align}

In the remaining part of the proof we are going to show that after next $6$ iteration of the Rips machine the measure (width) of each band is multiplied by $\lambda_{1}^2$. Indeed, one can check that after these $6$ iterations we obtain a band complex with the same combinatorial configuration of bands as for complex $Y$, and bases of a resulting complex can be expressed in terms of bases of complex $Y$ in the following way: 
$$\begin{pmatrix}
r'_{1}\\r'_{2}\\h'\\g'\\n'
\end{pmatrix}
=R_{1}
\begin{pmatrix}
r_{1}\\r_{2}\\h\\g\\n
\end{pmatrix},
$$
where $$R_{1}=\begin{pmatrix} 
8&2&4&-5&0\\
-2&5&0&2&-4\\
-2&-2&-1&1&1\\
4&2&2&-2&-1\\
-3&0&-2&2&0
\end{pmatrix}.$$

The lengths of the bands of the resulting complex (let us denote them by $l'_{1},l'_{2},l'_{3},l'_{4}$, respectively) can be expressed in terms of the lengths of bands of $Y$ in the following way:
$$\begin{pmatrix}
l'_{1}\\l'_{2}\\l'_{3}\\l'_{4}
\end{pmatrix}
=L_{1}
\begin{pmatrix}
l_{1}\\l_{2}\\l_{3}\\l_{4}
\end{pmatrix},
$$
where $$L_{1}=\begin{pmatrix} 
0&2&1&2\\
0&1&0&0\\
2&0&4&1\\
1&2&4&2
\end{pmatrix}.$$
The cycle is represented in the Figure \ref{3}.
\begin{figure}
\includegraphics[width=16cm,height=17cm]{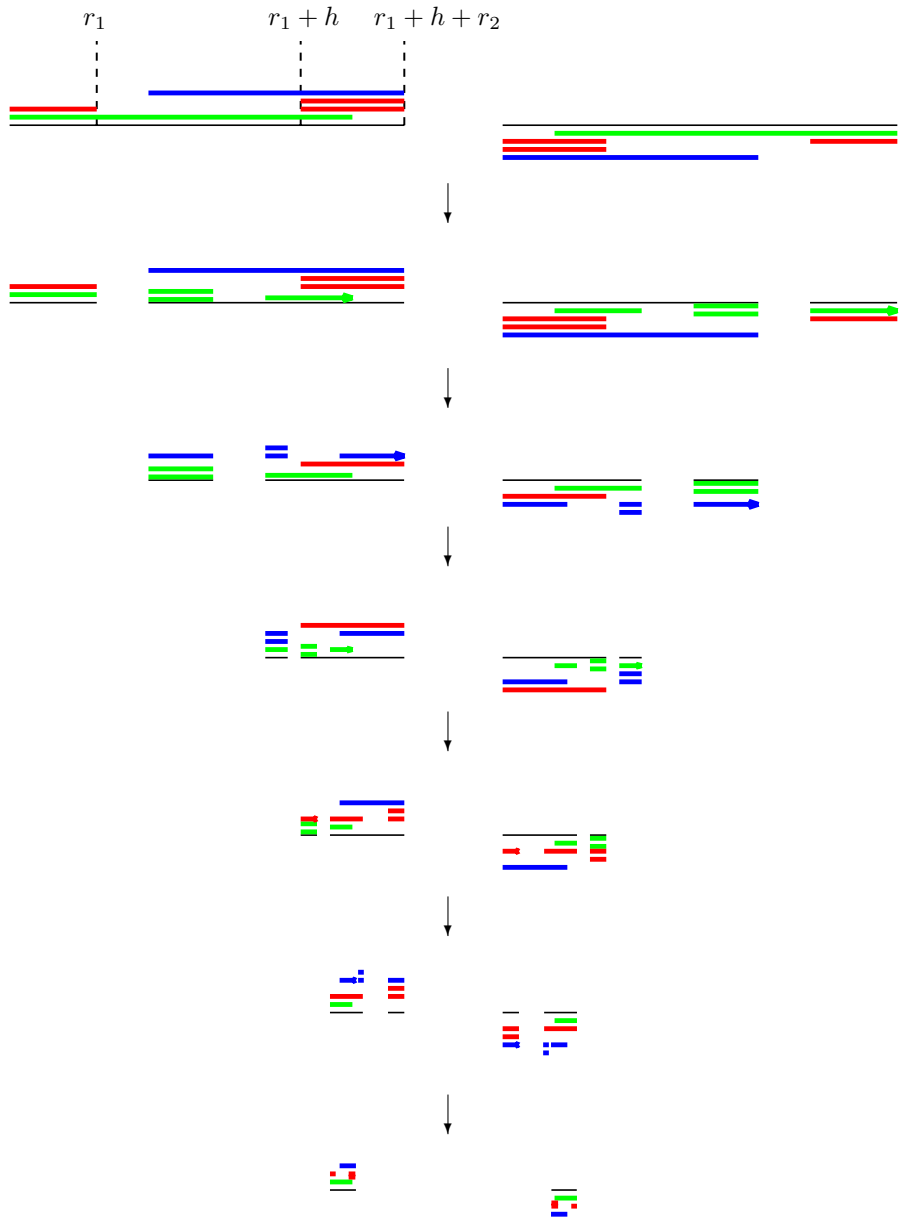}
\put(-282,430){\vector(0,-10){15}}
\put(-282,360){\vector(0,-10){15}}
\put(-282,300){\vector(0,-10){15}}
\put(-282,230){\vector(0,-10){15}}
\put(-282,160){\vector(0,-10){15}}
\put(-282,85){\vector(0,-10){15}}
\put(-420,490){$r_{1}$}
\put(-350,490){$r_{1}+h$}
\put(-310,490){$r_{1}+h+r_{2}$}
\caption{The Cycle of the Rips Machine: Example 1}
\label{3}
\end{figure}
\noindent Taking into account that $\lambda_{1}$ is a root of characterictic polynomial of matrix $N_{1}$ and so $$\lambda_{1}^4-\lambda_{1}^3-4\lambda_{1}^2+5\lambda_{1}-1=0,$$ in terms of polynomials with $\lambda_{1}$ and rational coefficients we have:
\begin{align}
r'_{1}& =5\lambda_{1}^{3}-\lambda_{1}^2-4\lambda_{1}+1 = r_{1}\lambda_{1}^2\\
r'_{2}& =-10\lambda_{1}^{3}+23\lambda_{1}^{2}-17\lambda_{1}+3 = r_{2}\lambda_{1}^2\\
h'& =5\lambda_{1}^{3}/4-13\lambda_{1}^2/2+13\lambda_{1}/2-5/4 = h\lambda_{1}^2\\
g'& =\lambda_{1}^3+4\lambda_{1}^2-5\lambda_{1}+1 = g\lambda_{1}^2\\
b'& =-7\lambda_{1}^{3}/2+5\lambda_{1}^2-3\lambda_{1}+1/2 = b\lambda_{1}^2
\end{align}

Now we only need to compare the smallest eigenvalue of $R_{1}$ (which is equal to $\lambda_{1}^2$) with the biggest eigenvalue of $L_{1}$ (let us denote it by $\mu_{1}$). One can check that $\lambda_{1}^2\approx0.0647$ and $\mu_{1}\approx6.1329$ and their product is strictly less than one. Therefore, almost all graphs $\Gamma_{x}$ have only one topological end.
\end{proof}
For finishing the proof of our main theorem, we prove the following 
\theoremstyle{theorem}
\begin{lemma}
If almost all graphs $\Gamma_{x}$ which correspond to orbits of IIS $S_{1}$ have exactly one topological end, then sections of the surface $\widehat M$ by almost any plane orthogonal to $H=(0,1,0)$ consists of exactly one connected component.
\end{lemma}
\begin{proof}
It is easy to check that the piecewise linear surface $M$ constructed above is always symmetric relative to the point $(3/10,(2a+c+b-u)/2,1/4)$ which implies that the foliation is the same in both parts into which the surface cuts the torus $\mathbb T^3$. Let $f$ be a function such that our surface is a level function of $f$. Let us consider particular plane section and suppose that it contains at least two unbounded components. Then $$\chi(\mathbb R^{2})=\chi(f<0)+\chi(f>0)-\chi(boundary)$$ and we have a contradiction because $\chi(\mathbb R^{2})=1,\chi(f<0)=\chi(f>0)=1$ due to the result about single topological end, but $\chi(boundary)>1$. 
\end{proof}

Analogically, we can prove the similar result for Dynnikov's example of chaotic section. Referring to the same argument about symmetry of the surface, we only need to prove the following.
\theoremstyle{theorem}

\begin{proposition} Let $\left(a,b,c,d,e\right)$ be an eigenvector
of the matrix $N_{2}$ with the eigenvalue $\lambda_{2}$ and positive coordinates.
Then for the corresponding interval identification system
\begin{equation*}
\begin{split}
S_{2}=(\left[0,a+b+c\right];& \left[0,a\right]\leftrightarrow\left[b+c,a+b+c\right],\\
                    & \left[0,b\right]\leftrightarrow\left[a+c,a+b+c\right], \\
                    & \left[d,d+c\right]\leftrightarrow\left[e,e+c\right])
\end{split}
\end{equation*}
almost all graphs $\Gamma_{x}$ have only one topological end.
\end{proposition}

\begin{proof}
As in the previous theorem, we start with a band complex, associated with the interval identification system $S_{2}$ and apply the Rips machine algorithm to this complex (let us denote it by $Z$). Note that $a$-intervals are represented by red band, $b$-intervals are represented by green band and $c$-intervals are represented by blue bands. Note that $(a,b,c,d,e)$ is an eigenvalue of matrix $N_{2}$ which corresponds to the eigenvalue $\lambda_{2}$ and so all these numbers can be represented as polynomials of $\lambda_{2}$ with rational coefficients. Note also that $\lambda_{2}$ is a non-rational root of characteristic polynomial of $N_{2}$ and $\pm1$ are also roots of the same polynomial and so $\lambda_{2}$ satisfies the following equation:
$$\lambda_{2}^3+8\lambda_{2}^2+12\lambda_{2}-1=0.$$

One can check that we have the following:
\begin{align}
a& =-10\lambda_{2}^{2}/3-58\lambda/3+2\\
b& =2\lambda_{2}^{2}/3+11\lambda_{2}/3\\
c& =8\lambda_{2}^{2}/3+47\lambda_{2}/3-1\\
d& =7\lambda_{2}^{2}/3+41\lambda_{2}/3-2/3\\
e& =-10\lambda_{2}^{2}/3-59\lambda_{2}/3+5/3
\end{align}

First $7$ iterations of the Rips machine result in a new band complex which consists of one support interval and three bands attached. Let us denote widths of bands of this complex by $a'$ for the red band, $b'$ for the green band and $c'$ for the blue band. Lengths of the parts of the support interval between the first points of $a'$-interval and $c'$-interval will be denoted by $d'$ and $e'$, respectively (see Figure for the details). Vertical lengths of three bands will be denoted by $l_{1}$ for the red band, $l_{2}$ for the green band and $l_{3}$ for the blue band.
One can check that $l_{1}=15, l_{2}=14, l_{3}=15$ and
$$\begin{pmatrix}
a'\\b'\\c'\\d'\\e'
\end{pmatrix}
=\begin{pmatrix} 
5&-9&-5&5&-5\\
-1&3&1&-2&2\\
-3&4&2&-1&1\\
5&-9&-4&4&-3\\
0&-2&-2&3&-2
\end{pmatrix}
\begin{pmatrix}
a\\b\\c\\d\\e
\end{pmatrix}.$$

Resulting band complex will be denoted by $Z'$. First $5$ iterations of the Rips machine are shown in Figure \ref{9}.
\begin{figure}
\includegraphics[width=20cm,height=14cm]{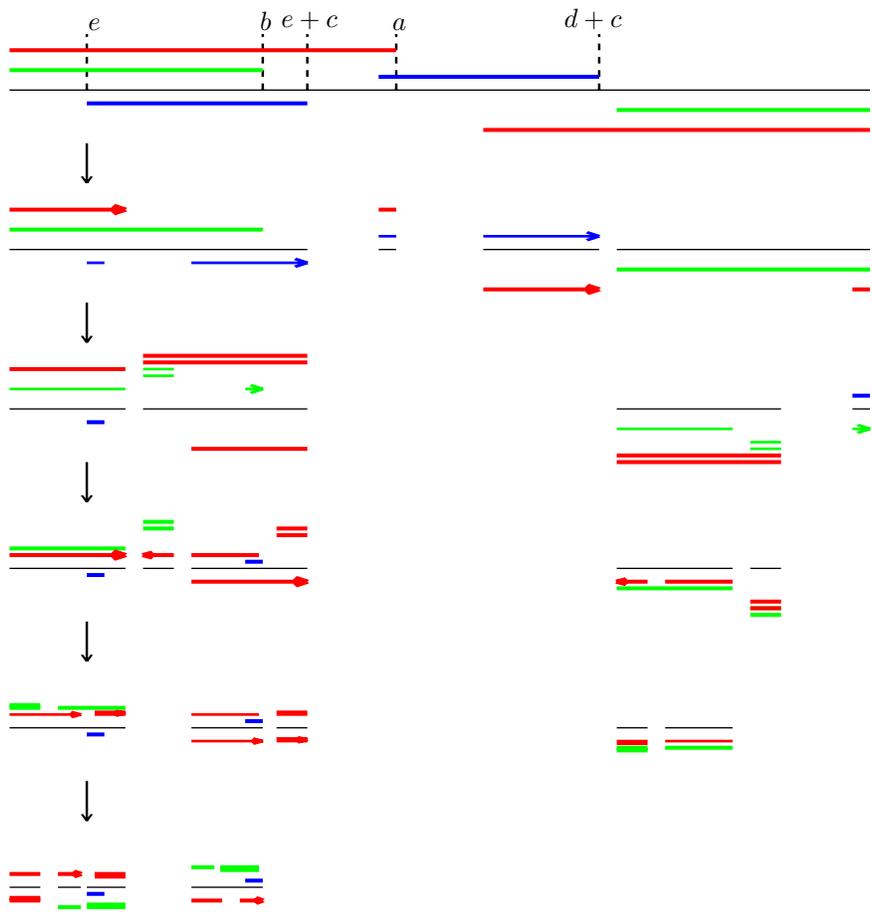}
\put(-415,400){$a$}
\put(-465,400){$b$}
\put(-530,400){$e$}
\put(-457,401){$e+c$}
\put(-350,401){$d+c$}
\caption{First Steps of the Rips Machine: Example 2}
\label{9}
\end{figure}
In terms of polynomials with degrees of $\lambda_{2}$ and rational coefficients we have:
\begin{align}
a'& =-23\lambda_{2}^{2}/3-124\lambda_{2}/3+10/3\\
b'& =-10\lambda_{2}^{2}/3-62\lambda_{2}/3+5/3\\
c'& =37\lambda_{2}^{2}/3+212\lambda_{2}/3-17/3\\
d'& =-14\lambda_{2}^{2}-236\lambda_{2}/3+19/3\\
e'& =7\lambda_{2}^{2}+125\lambda_{2}/3-10/3
\end{align}
In the remaining part of the proof we are going to check that after next $5$ iterations of the Rips machine the measure (width) of each band is multiplied by $\lambda_{2}$ (see Figure \ref{5}).
\begin{figure}
\includegraphics[width=20cm,height=15cm]{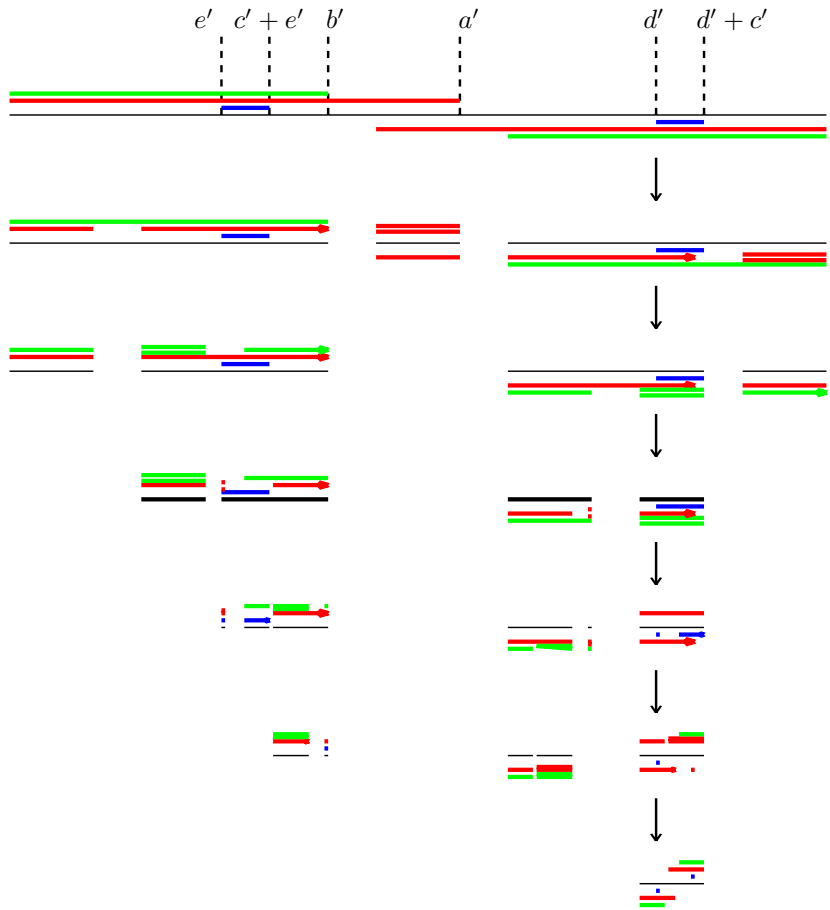}
\put(-390,410){$a'$}
\put(-440,410){$b'$}
\put(-475,410){$c'+e'$}
\put(-490,410){$e'$}
\put(-320,410){$d'$}
\put(-300,410){$d'+c'$}
\caption{The Cycle of the Rips Machine: Example 2}
\label{5}
\end{figure}
Indeed, one can check that these $5$ iterations result in a band complex with a same combinatorial configuration of bands as for complex $Z'$, and parameters of a resulting complex can be expressed in terms of parameters of complex $Z'$ in the following way:
$$\begin{pmatrix}
a''\\b''\\c''\\d''\\e''
\end{pmatrix}
=R_{2}
\begin{pmatrix}
a'\\b'\\c'\\d'\\e'
\end{pmatrix},$$where $$R_{2}=\begin{pmatrix} 
-5&5&1&1&0\\
1&-2&0&0&1\\
2&-2&-1&0&-1\\
-4&5&1&0&1\\
-2&1&-1&1&0
\end{pmatrix}.$$

It means that
\begin{align}
a''& =20\lambda_{2}^{2}+286\lambda_{2}/3+23/3 = a'\lambda_{2}\\
b''& =6\lambda_{2}^{2}+125\lambda_{2}/3-10/3= b'\lambda_{2}\\
c''& =-28\lambda_{2}^{2}-461\lambda_{2}/3+37/3= c'\lambda_{2}\\
d''& =100\lambda_{2}^{2}/3+523\lambda_{2}/3-14= d'\lambda_{2}\\
e''& =-43\lambda_{2}^{2}/3-262\lambda_{2}/3+7= e'\lambda_{2}
\end{align}

The lengths of bands of the resulting complex (let us denote them by $l_{1}'', l_{2}'', l_{3}''$) are expressed in terms of the lengths of bands of $Z'$ in the following way:
$$\begin{pmatrix}
l_{1}''\\l_{2}''\\l_{3}''
\end{pmatrix}
=L_{2}
\begin{pmatrix}
l_{1}\\l_{2}\\l_{3}
\end{pmatrix},$$where $$L_{2}=\begin{pmatrix} 
5 & 3 & 0\\
4 & 3 & 1\\
4 & 2 & 1
\end{pmatrix}.$$

Now we only need to compare the smallest eigenvalue of $R_{2}$ (which is equal to $\lambda_{2}$) with the biggest eigenvalue of $L_{2}$ (let us denote it by $\mu_{2}$). One can check that $\lambda_{2}\approx0.0798$ and $\mu_{2}\approx7.95$ and their product is strictly less than one. Therefore, almost all graphs $\Gamma_{x}$ have only one topological end.
\end{proof}

\emph{Acknowledgements.} I wish to thank I. Dynnikov for posing the problem and for constant attention to this work and T. Coulbois for explaining some of the details of the Rips theory to me. I am also very grateful to A. Zorich for making a lot of suggestions and improvements in an earlier version of the paper.

\end{document}